\documentclass[11pt]{amsart}
\usepackage[british]{babel}

\usepackage{amssymb,latexsym}
\usepackage{comment}
\usepackage{titletoc}
\usepackage{amsfonts, amsmath, amsthm, amssymb, fancybox, graphicx, color, times, babel}
\usepackage[utf8]{inputenc}
\usepackage{multirow}
\usepackage{etoolbox}

\usepackage[usenames,dvipsnames]{xcolor}

\usepackage[all,cmtip]{xy}
\usepackage[normalem]{ulem}

\newtheorem {thm}{Theorem}
\newtheorem {cor}[thm]{Corollary}
\newtheorem {lem}[thm]{Lemma}
\newtheorem {prop}[thm]{Proposition}
\theoremstyle{definition}
\newtheorem {rem}[thm]{Remark}

\DeclareMathOperator{\ord}{ord}
\DeclareMathOperator{\Gal}{Gal}

\DeclareMathOperator{\id}{id}

\DeclareMathOperator{\N}{N}
\DeclareMathOperator{\Li}{Li}
\DeclareMathOperator{\ind}{ind}

\newcommand{\Q}{\mathbb{Q}}

\newcommand{\Z}{\mathbb{Z}}

\newcommand{\p}{\mathfrak{p}}

\newcommand{\OO}{\mathcal{O}}

\newcommand{\av}[1]{\left\lvert#1\right\rvert}
\newcommand{\floor}[1]{\lfloor #1 \rfloor}
\newcommand{\bo}[1]{O\left( #1 \right)}
\newcommand{\cset}[1]{\left\lvert\left\{ #1 \right\}\right\rvert}
\newcommand{\set}[1]{\left\{ #1 \right\}}

\newcommand{\mc}{\mathcal}
\renewcommand{\geq}{\geqslant}
\renewcommand{\leq}{\leqslant}

\makeatletter
\newcommand{\customlabel}[2]{%
   \protected@write \@auxout {}{\string \newlabel {#1}{{#2}{\thepage}{#2}{#1}{}} }%
   \hypertarget{#1}{}
}
\makeatother

\usepackage{xr-hyper}
\usepackage{hyperref}

\title[On the distribution of the order and index]{On the distribution of the order and index\\ for the reductions of algebraic numbers}

\begin{document}

\author{Pietro~Sgobba}
\address[]{Department of Mathematics, University of Luxembourg, 6 Avenue de la Fonte, 4364 Esch-sur-Alzette, Luxembourg}
\email{pietro.sgobba@uni.lu}


\begin{abstract}
Let $\alpha_1,\ldots,\alpha_r$ be algebraic numbers in a number field $K$ generating a subgroup of rank $r$ in $K^\times$. We investigate under GRH the number of primes $\p$ of $K$ such that each of the orders of $(\alpha_i\bmod\p)$ lies in a given arithmetic progression associated to $\alpha_i$. We also study the primes $\p$ for which the index of $(\alpha_i\bmod\p)$ is a fixed integer or lies in a given set of integers for each $i$. An additional condition on the Frobenius conjugacy class of $\p$ may be considered. Such results are generalizations of a theorem of Ziegler from 2006, which concerns the case $r=1$ of this problem.
\end{abstract}

\maketitle

\section{Introduction}
Consider a number field $K$ and finitely many algebraic numbers $\alpha_1,\ldots,\alpha_r\in K^\times$ which generate a multiplicative subgroup of $K^\times$ of positive rank $r$.
Let $\p$ be a prime of $K$ such that for each $i$ the reduction of $\alpha_i$ modulo $\p$ is a well-defined element of $k_\p^\times$ (where $k_\p$ is the residue field at $\p$). We study the set of primes such that for each $i$ the multiplicative order of $(\alpha_i\bmod\p)$ lies in a given arithmetic progression. 

More precisely, write $\ord_{\p}(\alpha_i)$ for the order of $(\alpha_i\bmod\p)$. We will prove under GRH the existence of the density of primes $\p$ satisfying $\ord_{\p}(\alpha_i)\equiv a_i\bmod d_i$ for each $i$, where $a_i,d_i$ are some fixed integers. In Theorem \ref{main} we give an asymptotic formula for the number of such primes. 
We also study the density of primes satisfying conditions on the index. Write $\ind_\p(\alpha_i)$ for the index of the subgroup generated by $(\alpha_i\bmod\p)$ in $k_\p^\times$. Notice that $\ind_\p(\alpha_i)=(\N\p-1)/\ord_\p(\alpha_i)$. We prove the existence of the density of primes $\p$ such that  $\ind_\p(\alpha_i)=t_i$ for each $i$, where the $t_i$'s are positive integers, and more generally such that $\ind_\p(\alpha_i)$ lies in a given sequence of integers.
Given a finite Galois extension of $K$, a condition on the conjugacy class of Frobenius automorphisms of the primes lying above $\p$  may also be introduced.

These results are generalizations of Ziegler's work \cite{zieg} from 2006, which concerns the case of rank $1$. 
Moreover, in \cite{PS1} the author and Perucca have generalized Ziegler's results to study the set of primes for which the order of the reduction of a finitely generated group of algebraic numbers lies in a given arithmetic progression, and in \cite{PS2} they have investigated properties of the density of this set.
Notice that problems of this kind have been studied in various papers by Chinen and Murata, and by Moree, see for instance \cite{ChinenMurata, Moree}, and that they are related to Artin's Conjecture on primitive roots, see the survey \cite{MoreeArtin} by Moree.

\subsection{Notation}\label{notation}
We make use of the following standard notation: $\mu$ is the M\"obius function; $\varphi$ is Eluer's totient function; $\zeta_n$ is a primitive $n$-th root of unity; $(x,y)$ is the greatest common divisor of $x$ and $y$, while $[x_1,\ldots,x_r]$ is the least common multiple of $x_1,\ldots,x_r$; if $S$ is a set of primes of $K$, then $S(x)$ is the number of elements of $S$ having norm at most $x$. 
We write $N=(n_1,\ldots,n_r)$ and $T=(t_1,\ldots,t_r)$ for $r$-dimensional multi-indices, by which we mean $r$-tuples of positive integers.
We thus write
\[ \sum_N=\sum_{n_1\geq1}\cdots\sum_{n_r\geq1} \]
for the multiple series on the indices $n_i$, and similarly for $T$, as well as for finite multiple sums.
We denote by $K_{N,T}$ the compositum of the fields
\[ K\big( \zeta_{n_it_i},\alpha_i^{1/n_it_i} \big) \]
for $i\in\{1,\ldots,r\}$, namely
\[ K_{N,T}=K\big(\zeta_{[n_1t_1,\ldots,n_rt_r]},\alpha_1^{1/n_1t_1},\ldots,\alpha_r^{1/n_rt_r}\big)\,, \]
and we similarly define $F_{N,T}$, if $F$ is a finite extension of $K$. Moreover, if $F/K$ is Galois and $\p$ is a prime of $K$ which is unramified in $F$, then $(\p,F/K)$ denotes the conjugacy class of $\Gal(F/K)$  consisting of Frobenius automorphisms associated to the primes of $F$ lying above $\p$.

\subsection{Main results}\label{sec-main}
The following results are conditional under (GRH), by which we mean the extended Riemann hypothesis for the Dedekind zeta function of a number field, which allows us to use the effective Chebotarev density theorem (see for instance \cite[Theorem 2]{zieg}).

In the following statements we tacitly exclude the finitely many primes $\p$ of $K$ that appear in the prime factorizations of the fractional ideals generated by the $\alpha_i$'s, and those that ramify in a given finite Galois extension $F$ of $K$.

\begin{thm}\label{main} 
Let $F/K$ be a finite Galois extension, and  let $C$ be a union of conjugacy classes of $\Gal(F/K)$.
For $1\leq i\leq r$, let $a_i$ and $d_i\geq2$ be integers. Define the following set of primes of $K$:
\[
\mc P :=\set{\p: \ord_\p(\alpha_i)\equiv a_i\bmod d_i\ \forall  i ,\, \begin{pmatrix}\p \\ F/K  \end{pmatrix}\subseteq C  }\,.
\]

Assuming (GRH), we have
\begin{equation}\label{eq-P}
\mc P(x)=\frac{x}{\log x}\sum_{T}\
	    	\sum_{N } 
	    	\frac{(\prod_{i}\mu(n_i))c(N,T)}{[F_{w,N,T}:K]}
			+\bo{\frac{x}{(\log x)^{1+\frac{1}{r+1}}}}\,. 
\end{equation}
where $w=w(T):=[d_1t_1,\ldots,d_rt_r]$, and $F_{w,N,T}$ denotes the compositum of the fields $F(\zeta_w)$ and $F_{N,T}$, namely
\[
F_{w,N,T}=F\big(\zeta_{[d_1t_1,\ldots,d_rt_r,n_1t_1,\ldots,n_rt_r]},\alpha_1^{1/n_1t_1},\ldots,\alpha_r^{1/n_rt_r}\big)\,,
\]
and
\[
c(N,T)=\cset{\sigma\in\Gal(F_{w,N,T}/K):\,\forall i\ \sigma(\zeta_{d_it_i})=\zeta_{d_it_i}^{1+a_it_i}	,\, \sigma|_{K_{N,T}}=\id,\,\sigma|_F\in C}\,.
\]
In particular, $c(N,T)$ is nonzero only if $(1+a_it_i,d_i)=1$ and $(d_i,n_i)\mid a_i$ for all $i\in\{1,\ldots,r\}$. The constant implied by the $O$-term depends only on $K$, $F$, the $\alpha_i$'s and the $d_i$'s.
\end{thm}

\begin{rem}
The multiple series involved in the asymptotic formula \eqref{eq-P} converges. This statement is a consequence of the results of Section \ref{sec-euler}: more precisely, the convergence follows by Proposition \ref{prop-PS1} and Theorem \ref{pollack}, and we prove this property in Corollary \ref{cor-convergence2}. Notice that the same remark applies to the series in the formulas \eqref{r-asymptotic} and \eqref{eq-S} below (see also Corollary \ref{cor-convergence1}).
\end{rem}

\begin{thm}\label{thm-index1} 
Let $F/K$ be a finite Galois extension, and let $C$ be a union of conjugacy classes of $\Gal(F/K)$. 
Let $T=(t_1,\ldots,t_r)$ be an $r$-tuple of positive integers.  
Define the following set of primes of $K$:
\begin{equation}
\mc R :=\set{\p: \ind_\p(\alpha_i)=t_i\ \forall i,\, \begin{pmatrix}\p \\ F/K  \end{pmatrix}\subseteq C  }\,.
\label{r-set}
\end{equation}

Assuming (GRH), and supposing that $x\geq t_i^3$  for all $i$, we have
\begin{equation}\label{r-asymptotic}
\mc R(x)= \frac{x}{\log x}\sum_{N}\frac{(\prod_{i}\mu(n_i))c'(N,T)}{[F_{N,T}:K]} +\bo{\frac{x}{\log^{2}x}+\sum_{i=1}^r\frac{x\log\log x}{\varphi(t_i)\log^2x}},
\end{equation}
where 
\begin{equation}\label{c-fct}
c'(N,T)=\cset{\sigma\in\Gal(F_{N,T}/K):\, \sigma|_{K_{N,T}}=\id,\,\sigma|_F\in C}\,.
\end{equation}
The constant implied by the $O$-term  depends only on $K,F$ and the $\alpha_i$'s.
\end{thm}

Notice that applying Theorem \ref{thm-index1} with $t_i=1$ for all $i$ and $F=K$ (which gives $c'(N,T)=1$ for all $N$) yields a multidimensional variant of Artin's Conjecture on primitive roots over number fields.

\begin{thm}\label{thm-index} 
Let $F/K$ be a finite Galois extension, and let $C$ be a union of conjugacy classes of $\Gal(F/K)$. 
For $1\leq i\leq r$, let $S_i$ be some nonempty sets of positive integers.
Define the following set of primes of $K$:
\[
\mc S :=\set{\p: \ind_\p(\alpha_i)\in S_i\ \forall i ,\, \begin{pmatrix}\p \\ F/K  \end{pmatrix}\subseteq C  }\,.
\]

Assuming (GRH), we have
\begin{equation}\label{eq-S}
\mc S(x)=\frac{x}{\log x}\sum_{\substack{T\\ t_i\in S_i }}
	    	\sum_{N } 
	    	\frac{(\prod_{i}\mu(n_i))c'(N,T)}{[F_{N,T}:K]}
			+\bo{\frac{x}{(\log x)^{1+\frac{1}{r+1}}}}\,,
\end{equation}
where
\[ c'(N,T)=\cset{\sigma\in\Gal(F_{N,T}/K):\, \sigma|_{K_{N,T}}=\id,\,\sigma|_F\in C}\,. \]
The constant implied by the $O$-term depends only on $K$, $F$ and the $\alpha_i$'s.
\end{thm}
In particular, we may choose $S_i$ to be the set  of positive integers lying in an arithmetic progression, say for instance $\{k\geq1:k\equiv a_i\bmod d_i\}$, with $a_i$ and $d_i\geq2$ integers.

Notice that if we take $r=1$ in Theorems \ref{main} and \ref{thm-index1}, then we obtain the same formulas as in \cite[Theorem 1, Proposition 1]{zieg}, respectively.

\subsection{Overview}
In order to generalize  Zielger's proofs \cite{zieg} to obtain Theorems \ref{main}, \ref{thm-index1} and \ref{thm-index}, some crucial results are needed. The first one is Theorem \ref{pollack} which is an estimate for a multiple series involving the Euler's totient function $\varphi$. Section \ref{sec-euler} is devoted to the proof of this theorem. The other two results concern Kummer extensions of number fields and are proven in Section \ref{sec-kummer}. More precisely, Proposition \ref{prop-PS1} states that the failure of maximality of their degree is bounded in a strong way, whereas Proposition \ref{prop-discr}  gives an estimate for their discriminant. All these results will be used to deal with the asymptotics of the sets of primes considered in Section \ref{sec-main}.

In Section \ref{sec-index1} we prove Theorem \ref{thm-index1}, and then we use this result in Section \ref{sec-index2} to set more general conditions on the index and achieve Theorem \ref{thm-index}. In Section \ref{sec-order} we prove Theorem \ref{main} by transforming the conditions on the order into conditions on the index and on the Frobenius conjugacy class with respect to certain finite Galois extensions,  thus allowing to apply the previous results.

\section{On the Euler's totient function}\label{sec-euler}
In this section we estimate some expressions involving the Euler's totient function. We keep the notation introduced in Section \ref{notation}. In particular, we  write $N=(n_1,\ldots,n_r)$ for a multi-index (whose components are positive integers). Also we denote by $\tau(n)$ the number of positive divisors of $n$.

Recall the following well-known estimates:
\begin{equation}\label{tau}
\tau(n)=\bo{n^\varepsilon} \quad \forall \varepsilon>0 \,,
\end{equation}
(see \cite[Formula (2.20)]{MultNT});
\begin{equation}\label{phi2}
\frac{n}{\varphi(n)}=\bo{n^\varepsilon} \quad \forall \varepsilon>0\,,
\end{equation}
which follows by noticing that for each prime $p$ there is a constant $c_\varepsilon>0$ such that $(1-1/p)\geq c_\varepsilon/p^\varepsilon$, and we may take $c_\varepsilon=1$ for all $p$ sufficiently large (with respect to $\varepsilon$);
\begin{equation}\label{phi1}
\sum_{n\leq x}\frac{n}{\varphi(n)}=\bo{x}\,,
\end{equation}
(see for instance \cite[Formula (2.32)]{MultNT}).

Our goal is to prove a multidimensional variant of the estimate $\sum_{n>x}\frac{1}{\varphi(n)n}=O\big(\frac{1}{x}\big)$  
(see for instance \cite[Lemma 7]{zieg}), namely
\begingroup
\makeatletter
\apptocmd{\thethm}{\unless\ifx\protect\@unexpandable@protect\protect\footnote{This result and its proof are due to Paul Pollack, see also the acknowledgments.}\fi}{}{}
\makeatother
\begin{thm}
\label{pollack}
We have
\begin{equation}\label{pollack-eq}
\sum_{\substack{N\\ n_1>x}}
	\frac{1}{\varphi([n_1,n_2,\ldots,n_r])n_1n_2\cdots n_r}=\bo{\frac{1}{x}}\,.
\end{equation}
\end{thm}
\endgroup

\begin{lem}\label{phi-lem1}
Let $z$ be a positive integer, then for every $\varepsilon>0$ we have
\[ \sum_{n\leq x}(n,z)\cdot\frac{n}{\varphi(n)}=\bo{xz^\varepsilon}\,. \]
\end{lem}

\begin{proof}
We have
\begin{align*}
&\sum_{n\leq x}(n,z)\cdot\frac{n}{\varphi(n)}
	=\sum_{d\mid z}\sum_{\substack{n\leq x\\ (n,z)=d}}d\cdot\frac{n}{\varphi(n)} \\
	\leq& \sum_{d\mid z} d\cdot \sum_{m\leq x/d} \frac{md}{\varphi(md)} 
	 \leq \sum_{d\mid z} \frac{d^2}{\varphi(d)} \sum_{m\leq x/d} \frac{m}{\varphi(m)}\,.
\end{align*}
Then the formula \eqref{phi1} yields
\[ \sum_{n\leq x}(n,z)\cdot\frac{n}{\varphi(n)}=O\Bigg(x\sum_{d\mid z} \frac{d}{\varphi(d)}\Bigg)\,. \]
We may then conclude by using  \eqref{tau} and \eqref{phi2} with $\varepsilon/2$ to get 
\[ \sum_{d\mid z} \frac{d}{\varphi(d)}=\bo{\tau(z)\cdot\max_{d\mid z}\frac{d}{\varphi(d)}}
										=\bo{z^\varepsilon}\,. \qedhere  \]
\end{proof}

\begin{lem}\label{phi-lem2}
We have
\[ \sum_{\substack{N\\ n_1\leq x }}
\frac{n_1}{\varphi([n_1,\ldots,n_r])n_2\cdots n_r}=\bo{x}\,. \]
\end{lem}

\begin{proof}
We may assume $r\geq2$, the case $r=1$ being just \eqref{phi1}.
We will make use of the formula $\varphi(n)=n\prod_{p\mid n}(1-1/p)$, where $p$ denotes a prime number. The main term of the considered series can thus be written as
\[ \frac{n_1}{[n_1,\ldots,n_r]n_2\cdots n_r}\prod_{p\mid[n_1,\ldots,n_r]}\left( 1-\frac{1}{p} \right)^{-1}\,. \]
Then, in view of the identity $[n_1,\ldots,n_r]\cdot(n_1,[n_2,\ldots,n_r])=n_1[n_2,\ldots,n_r]$, we can bound our series from above by
\begin{align*}
&\sum_{\substack{N\\ n_1\leq x }}
\frac{(n_1,[n_2,\ldots,n_r])}{[n_2,\ldots,n_r]n_2\cdots n_r}
\prod_{i=1}^r\  \prod_{p\mid n_i}\left( 1-\frac{1}{p} \right)^{-1} \\
= &  \sum_{n_2,\ldots,n_r\geq1}
\frac{1}{[n_2,\ldots,n_r]n_2\cdots n_r}
\prod_{i=2}^r\frac{n_i}{\varphi(n_i)}\cdot\sum_{n_1\leq x}(n_1,[n_2,\ldots,n_r])\frac{n_1}{\varphi(n_1)}\,.
\end{align*}
Taking $n=n_1$, $z=[n_2,\ldots,n_r]$ and $\varepsilon=1/2$, Lemma \ref{phi-lem1} says that the inner sum  is estimated by $\bo{x[n_2,\ldots,n_r]^{1/2}}$. Applying the obvious inequalities
\[ [n_2,\ldots,n_r]\geq(n_2\cdots n_r)^{1/{(r-1)}}\geq(n_2\cdots n_r)^{1/r}\,, \]
we can then estimate the series by
\begin{align*}
&\ \bo{x\sum_{n_2,\ldots,n_r\geq1}\frac{1}{(n_2\cdots n_r)^{1+1/2r}}\cdot\prod_{i=2}^r\frac{n_i}{\varphi(n_i)}} \\
=& \ \bo{x\sum_{n_2,\ldots,n_r\geq1}\frac{1}{(n_2\cdots n_r)^{1+1/4r}}}\\
=& \ \bo{x\left(\zeta\Big(1+\frac{1}{4r}\Big)\right)^{r-1}}=\bo{x}\,,
\end{align*}
where in the first equality we used the estimates $\frac{n_i}{\varphi(n_i)}=O(n_i^{1/4r})$ in view of \eqref{phi2}, and for the last equality we used the fact that the Riemann zeta function $\zeta$ is convergent at $(1+1/4r)$.
\end{proof}

We are now ready to prove Theorem \ref{pollack}. 

\begin{proof}[Proof of Theorem \ref{pollack}]
We may assume $r\geq2$. We decompose the series considered in \eqref{pollack-eq} into the sums over $n_1$ lying in dyadic intervals, i.e.\ we express it as
\begin{equation}\label{eq2-pollack}
\sum_{j\geq0}\ \sum_{\substack{N \\ 2^jx<n_1\leq2^{j+1}x}}
	\frac{1}{\varphi([n_1,\ldots,n_r])n_1\cdots n_r}\,.
\end{equation}
We now estimate each inner series on the multi-indices $N$ in \eqref{eq2-pollack}. For $j\geq0$, each of them equals
\begin{align*}
&\sum_{2^jx<n_1\leq2^{j+1}x}\frac{1}{n_1^2} \sum_{n_2,\ldots,n_r\geq1}
	\frac{n_1}{\varphi([n_1,\ldots,n_r])n_2\cdots n_r} \\
&\leq  \frac{1}{(2^jx)^2}\cdot\sum_{\substack{N \\ n_1\leq 2^{j+1}x}}
	\frac{n_1}{\varphi([n_1,\ldots,n_r])n_2\cdots n_r} \\
	&= \bo{\frac{1}{(2^jx)^2}\cdot 2^{j+1}x}=\bo{\frac{1}{2^jx}}\,,
\end{align*}
where the estimate is due to Lemma \ref{phi-lem2}. Finally, we conclude by summing the obtained error terms over $j$, so that \eqref{eq2-pollack} equals $\bo{1/x}$.
\end{proof}

The following result is an immediate consequence of Theorem \ref{pollack}.
\begin{cor}
Let $x_1,\ldots,x_r\geq1$. Then we have
\[ \sum_{\substack{N \\ n_i>x_i}}
	\frac{1}{\varphi([n_1,n_2,\ldots,n_r])n_1n_2\cdots n_r}=\bo{\frac{1}{\max_{i}(x_i)}}\,. \]
\end{cor}

\begin{proof}
Up to swapping the variables, we may suppose that $x_1=\max_i(x_i)$ and apply Theorem \ref{pollack}.
\end{proof}

\section{Kummer theory for number fields}\label{sec-kummer}
Let $K$ be a number field, and let $\alpha_1,\ldots, \alpha_r$ be algebraic numbers which  generate a  multiplicative subgroup $G$ of $K^\times$ of positive rank $r$. Notice that $G$ is torsion-free. In this section we prove some results about cyclotomic-Kummer extensions of $K$ of the type $K(\zeta_{n},\alpha_1^{1/t_1},\ldots, \alpha_r^{1/t_r})$ with $t_i\mid n$ for all $i$.

\subsection{Bounded failure of maximality for Kummer degrees}
In \cite[Theorem 3.1]{PS1} Perucca and the author showed, with a direct proof, that the failure of maximality of Kummer degrees of the type $[K(\zeta_m,G^{1/n}):K(\zeta_m)]$, with $n\mid m$, is bounded in a strong way. The following result is a further generalization and a consequence of this fact.

\begin{prop}\label{prop-PS1}
There exists an integer $B\geq1$, which depends only on $K$ and the $\alpha_i$'s,  such that for all positive integers $n,t_1,\ldots, t_r$, where $n$ is a common multiple of the $t_i$'s, we have 
\begin{equation}\label{bound}
\frac{\prod_{i=1}^rt_i}{\big[K\big(\zeta_{n},\alpha_1^{1/t_1},\ldots, \alpha_r^{1/t_r}\big):K(\zeta_n)\big]}\mid B\,.
\end{equation}
\end{prop}

\begin{proof}
Let $n,t_1,\ldots, t_r$ be arbitrary with $t:=[t_1,\ldots, t_r]\mid n$. Then by \cite[Theorem 3.1]{PS1} there is $B\geq1$ (depending only on $K$ and the $\alpha_i$'s) such that 
\begin{equation}\label{bound_PS1}
\frac{t^r}{\big[K\big(\zeta_{n},\alpha_1^{1/t},\ldots, \alpha_r^{1/t}\big):K(\zeta_n)\big]}\mid B\,.
\end{equation}
We show that this bound $B$ satisfies also \eqref{bound}. We have
\[ K\big(\zeta_{n},\alpha_1^{1/t_1},\ldots, \alpha_r^{1/t_r}\big)
\subseteq K\big(\zeta_{n},\alpha_1^{1/t},\ldots, \alpha_r^{1/t}\big) \]
and the degree of this extension divides $t^r/\prod_it_i$ as $\alpha_i^{1/t}=(\alpha_i^{1/t_i})^{t
_i/t}$ (up to a $t$-th root of unity) for every $i$. We deduce that the ratio in \eqref{bound} is a divisor of the ratio in \eqref{bound_PS1}, and hence it divides $B$.
\end{proof}

Notice that the bound $B$ considered in the proof is not optimal in general, but it is suitable for our purposes.

\begin{cor}\label{cor-bound}
For all positive integers $n,t_1,\ldots, t_r$, where $n$ is a common multiple of the $t_i$'s, we have 
\[
\big[K\big(\zeta_{n},\alpha_1^{1/t_1},\ldots, \alpha_r^{1/t_r}\big):K\big]
\geq \frac{\varphi(n)\prod_it_i}{[K:\Q]B}\,,
\]
where $B\geq1$ is the integer from Proposition \ref{prop-PS1} associated to $K$ and the $\alpha_i$'s.
\end{cor}

\begin{proof}
By \eqref{bound} the degree of the considered Kummer extension over $K(\zeta_n)$ is at least $\prod_it_i/B$, whereas it is easy to see that $[K(\zeta_n):K]\geq \varphi(n)/[K:\Q]$.
\end{proof}

Recall the notation for the fields $K_{N,T}$ and $K_{w,N,T}$ from Section \ref{notation} and Theorem \ref{main}, where $N,T$ are $r$-tuples of positive integers and $w=w(T):=[d_1t_1,\ldots,d_rt_r]$ for some positive integers $d_1,\ldots,d_r$. 
\begin{cor}\label{cor-convergence1}
Fix an $r$-tuple $T$. The series $\sum_{N}\frac{1}{[K_{N,T}:K]}$ converges.
\end{cor}
In particular, the series of this type, which we will consider in the later sections, converge absolutely and $\sum_N=\sum_{n_1\geq1}\cdots\sum_{n_r\geq1}$ can be interpreted as the series over the multi-indices $N$.

\begin{proof}
By Corollary \ref{cor-bound} we can bound
\[ \frac{1}{[K_{N,T}:K]}\leq\frac{[K:\Q]B}{\varphi([n_1,\ldots,n_r])n_1\cdots n_r}\,. \]
Then the convergence follows by Theorem \ref{pollack}.
\end{proof}

\begin{cor}\label{cor-convergence2}
The series $\sum_T\sum_{N}\frac{1}{[K_{w,N,T}:K]}$ converges.
\end{cor}
In particular, the series of this type, which we will consider in the later sections, converge absolutely and $\sum_T\sum_N=\sum_{t_1\geq1}\cdots\sum_{t_r\geq1}\sum_{n_1\geq1}\cdots\sum_{n_r\geq1}$ can be interpreted as the series over the multi-indices $T$ and $N$.
 
\begin{proof}
Since the degree $[K_{N,T}:K]$ divides $[K_{w,N,T}:K]$ and $[n_1t_1,\ldots,n_rt_r]$ is divisible by $[n_1,t_1,\ldots,n_r,t_r]$, applying Corollary \ref{cor-bound} we can bound 
\[ \frac{1}{[K_{w,N,T}:K]}\leq \frac{1}{[K_{N,T}:K]}\leq \frac{[K:\Q]B}{\varphi([n_1,t_1,\ldots,n_r,t_r])n_1t_1\cdots n_rt_r}\,. \]
The convergence follows again by Theorem \ref{pollack}.
\end{proof}

\subsection{Estimates for the discriminant}

We now prove an estimate for the discriminant of a cyclotomic-Kummer extension of the type $K(\zeta_{n},\alpha_1^{1/t_1},\ldots, \alpha_r^{1/t_r})$. In fact, we give a variant of \cite[Theorem 4.2]{PS1}.

We write $\OO_K$ for the ring of integers of $K$.
If $L/K$ is a finite extension of number fields, we denote by $\N_{L/K}$ the relative norm for fractional ideals of $L$, by $d_{L/K}$ the relative discriminant, and by $d_K$ the absolute discriminant of $K$. 
We will make use of the following relation for the relative discriminants of a tower of number fields $K''/K'/K$ (see for instance \cite[Ch.\ III, Corollary 2.10]{neukirch}):  
\begin{equation}\label{chaindisc}
d_{K''/K}=\N_{K'/K}(d_{K''/K'})\cdot d_{K'/K}^{[K'':K']}\,.
\end{equation}

\begin{prop}\label{prop-discr} 
Let $K$ be a number field, and let $\gamma_1,\ldots,\gamma_r\in K^\times$ be algebraic numbers which are not roots of unity. Let $t_1,\ldots,t_r$ be positive integers and let $n$ be a common multiple of the $t_i$'s. Setting $F:=K(\zeta_n,\gamma_1^{1/t_1},\ldots,\gamma_r^{1/t_r})$, we have 
\[ \frac{\log\av{d_F}}{\varphi(n)\prod_{i}t_i}\leq [K:\Q]\cdot\log\Big(n\prod_{i}t_i\Big)+ O(1)\,.\]
For $1\leq i\leq r$, write $\gamma_i=\alpha_i/\beta_i$ with $\alpha_i,\beta_i\in\OO_K$. Then the constant implied by the $O$-term can be taken to be 
\[ \log\av{d_K} +2\sum_{i=1}^r\log\av{\N_{K/\Q}(\alpha_i\beta_i)}\,.\]
\end{prop}

\begin{proof}
Set $t:=\prod_{i=1}^rt_i$, and for $1\leq i\leq r$, write $L_i$ for the extension of $K$ generated by some fixed root $\sqrt[t_i]{\gamma_i}$. We first estimate the relative discriminant $d_{F/K}$. The field $F$ is the compositum of $K(\zeta_{n})$ and the fields $L_i$, so that in view of \cite[Lemma 4.1(3)]{PS1} and the inequalities $[F:K(\zeta_n)]\leq t$ and $[F:L_i]\leq \varphi(n)t/t_i$ for all $i$, we have  
\begin{equation}
d_{F/K}\mid (d_{K(\zeta_n)/K})^{t}\cdot \prod_{i=1}^r (d_{L_i/K})^{\varphi(n)t/t_i} \,.
\label{ld1}\end{equation}
As for the relative discriminants of the extensions $K(\zeta_n)/K$ and $L_i/K$ we have the following estimates:
\[ d_{K(\zeta_n)/K}\mid n^{\varphi(n)}\OO_K \quad \text{ and } \quad
d_{L_i/K}\mid  (\alpha_i\beta_i)^{2t_i}t_i^{t_i}\mathcal O_K\,,  \]
(see the proof of \cite[Theorem 4.2]{PS1}, formulas (4.5) and (4.6), respectively). Combining these two divisibilities with \eqref{ld1} we obtain 
\begin{equation}\label{disc_FK}
d_{F/K} \mid \Bigg(n^{\varphi(n)t}\cdot \prod_{i=1}^r \Big((\alpha_i\beta_i)^{2t_i}t_i^{t_i}\Big)^{\varphi(n)t/t_i}\Bigg)\mathcal O_K
=\left((nt)^{\varphi(n)t}\cdot A^{2\varphi(n)t}\right)\mathcal O_K\,,
\end{equation}
where we set $A:=\prod_{i=1}^r \alpha_i\beta_i$.
In view of the identity \eqref{chaindisc}, we have the following formula for the absolute discriminant of $F$:
\[
\av{d_{F}}=\av{\N_{K/\Q}(d_{F/K})}\av{d_{K}}^{[F:K]}\,,
\]
where $\av{I}$ denotes the nonnegative generator of the $\Z$-ideal $I$. Hence, using \eqref{disc_FK} we can bound $\log\av{d_F}$ from above with the sum  of the following terms 
\[ \begin{array}{rcl}
\log\av{\N_{K/\Q}((nt)^{\varphi(n)t}\OO_K)} & = & \varphi(n)t\cdot [K:\Q] \cdot  \log(nt)\\
 \log\av{\N_{K/\Q}(A^{2\varphi(n)t}\OO_K)} & = & \varphi(n)t\cdot 2\log\av{\N_{K/\Q}(A)}\\
\log  \av{d_{K}}^{[F:K]} & \leqslant & \varphi(n)t \cdot \log\av{d_{K}}\,.\\
\end{array}\] 
We deduce
\[\frac{\log\av{d_{F}}}{\varphi(n)t}\leqslant 
[K:\Q] \log(nt)+\log\av{d_{K}}+2\log\av{\N_{K/\Q}(A)}\,. \qedhere
\]
\end{proof}

\section{The asymptotic formula for the index}\label{sec-index1}
The aim of this section is proving Theorem \ref{thm-index1} with the method of \cite[Section 3]{zieg}. We keep the notation of the introduction and, in particular, of Theorem \ref{thm-index1}. Recall that $N=(n_1,\ldots,n_r)$ and $T=(t_1,\ldots,t_r)$ are multi-indices (whose components are positive integers). 

Notice that throughout Sections \ref{sec-index1}-\ref{sec-order} we may assume $r\geq2$, as the case $r=1$ was proven in \cite{zieg}. Yet all our arguments also work for $r=1$.
Moreover, from now on we say that a prime $\p$ of $K$ is of \emph{degree} $1$ if it has ramification index and residue class degree over $\Q$ equal to $1$. When  necessary, thanks to \cite[Lemma 1]{zieg} we will estimate   the number of primes of $K$ which are not of degree $1$ by the error term $\bo{\sqrt{x}/\log x}$.

\begin{rem}\label{rem-inclexcl}
The defining conditions of the set $\mc R$  (see \eqref{r-set}), namely $\ind_\p(\alpha_i)=t_i$, are equivalent to: $t_i\mid\ind_\p(\alpha_i)$ and $q t_i\nmid\ind_\p(\alpha_i)$ for every prime $q$. With finitely many applications of the inclusion-exclusion principle, we get:
\[ \mc R(x)=\sum_{N}
	\bigg(\prod_{i=1}^r\mu(n_i)\bigg)
	\cdot\cset{\p :\N\p\leq x,\, \forall i\ n_it_i\mid\ind_\p(\alpha_i) ,\, \begin{pmatrix}\p \\ F/K  \end{pmatrix}\subseteq C}\,. \]
If $\p$ is of degree $1$, then by \cite[Lemma 2]{zieg} the condition $n_it_i\mid\ind_\p(\alpha_i)$ holds if and only if  $\p$ splits completely in $K(\zeta_{n_it_i},\alpha_i^{1/n_it_i})$. Moreover, $\p$ splits completely in each of these fields, for $1\leq i\leq r$, if and only if it splits completely in their compositum. Hence we can write $\mc R(x)$ as
\[ \sum_{N}
	\bigg(\prod_{i=1}^r\mu(n_i)\bigg)\cdot
	\av{\left\{\p: \N\p\leq x, \, 
		\begin{pmatrix}\p \\ K_{N,T}/K  \end{pmatrix}=\id,\, \begin{pmatrix}\p \\ F/K  \end{pmatrix}\subseteq C\right\}}+\bo{\frac{\sqrt{x}}{\log x}}.\]
\end{rem}

For real numbers $\xi,\eta\geq1$, fix an $r$-tuple $T$ and define the sets
\[ \mc M_\xi:=\set{\p:\forall i \ t_i\mid\ind_\p(\alpha_i)  \ \text{ and }\ t_iq\nmid \ind_\p(\alpha_i)\,\forall q<\xi \text{ prime},\, \begin{pmatrix}\p \\ F/K  \end{pmatrix}\subseteq C}\, , \]
\[ \mc M_{\xi,\eta}:=\set{\p:  t_iq\mid\ind_\p(\alpha_i)\text{ for some }i \text{ and some } \xi\leq q<\eta\text{ prime},\, \begin{pmatrix}\p \\ F/K  \end{pmatrix}\subseteq C} \]
(where we tacitly exclude the finitely many primes $\p$ of $K$ appearing in the factorization of the $\alpha_i$'s or ramifying in $F$).

Since for $\p$ with $\N\p\leq x$ we have $\ind_\p(\alpha_i)\mid\N\p-1<\floor{x}$,  it is clear that we have
\[
\mc R(x)=\mc M_{\floor{x}}(x)\,.
\]
Setting $\xi:=\frac{1}{6r}\log x$ and $\eta:=\floor{x}$, we have $\mc M_\eta(x)\leq \mc M_\xi(x)$. On the other hand, $\mc M_{\eta}(x)$ can be obtained by subtracting from $\mc M_\xi(x)$ the number of those primes $\p$ satisfying $t_iq\nmid \ind_\p(\alpha_i)$ for all $i$ and for all prime numbers $q<\xi$ but with $t_iq\mid\ind_\p(\alpha_i)$ for some $i$ and some prime $\xi\leq q<\eta$, so that
\[ \mc M_\eta(x)\geq \mc M_\xi(x)-\mc M_{\xi,\eta}(x)\,. \]
Therefore we get
\begin{equation}
\mc R(x)= \mc M_\xi(x)+O\big(\mc M_{\xi,\eta}(x)\big)\,.
\label{eq4}
\end{equation}

First we estimate the main term $\mc M_\xi(x)$.
\begin{lem}\label{lem-index-main} Assume (GRH). Let $x\geq t_i^{3}$ for all $i$. Then we have
\[ \mc M_\xi(x)=\frac{x}{\log x}\sum_{N}\frac{(\prod_{i}\mu(n_i))c'(N,T)}{[F_{N,T}:K]}+O\left(\frac{x}{\log^2 x}\right)\,, \]
where $c'(N,T)$ is defined in \eqref{c-fct}.\end{lem}

Notice that by definition, the coefficients $c'(N,T)$ are bounded by the size of $C$ and hence by $[F:K]$, independently of $N,T$.

\begin{proof} Denote by $E$ the set of the positive squarefree integers  which can be written as a product of primes $q$ with $q<\xi$. Applying the inclusion-exclusion principle as in Remark \ref{rem-inclexcl} yields
\[ \mc M_\xi(x)=\sum_{\substack{N \\ n_i\in E}}\bigg(\prod_{i=1}^r\mu(n_i)\bigg)\cdot
\cset{\p: \N\p\leq x,\,\begin{pmatrix}\p \\ F_{N,T}/K  \end{pmatrix}\subseteq C_{N,T}}+\bo{\frac{\sqrt{x}}{\log x}}\,,  \]
where $C_{N,T}$ is defined by
\[ C_{N,T}=\set{\sigma\in\Gal(F_{N,T}/K):\, \sigma|_{K_{N,T}}=\id,\,\sigma|_F\in C} \]
and has size $c'(N,T)$ (see \eqref{c-fct}).

Since we are assuming (GRH), by the effective Chebotarev Density Theorem (see for instance \cite[Theorem 2]{zieg}) the number of primes $\p$ of $K$ which are unramified  in $F_{N,T}$ and such that $\N\p\leq x$ and the Frobenius conjugacy class of $\p$ is contained in $C_{N,T}$ is given by (recalling that $c'(N,T)\leq [F:K]$)
\[ \Li(x)\frac{c'(N,T)}{[F_{N,T}:K]}
+O\left(\frac{\sqrt{x}\log \left(x^{[F_{N,T}:\Q]}\cdot\av{d_{F_{N,T}}} \right)}{[F_{N,T}:K]}\right)\,, \]
where $d_{F_{N,T}}$ is the absolute discriminant of $F_{N,T}$.
Then we write $\mc M_\xi(x)$ as the multiple sum  
\begin{equation}\label{eq-m}
\Li(x)\sum_{\substack{N\\ n_i\in E}}\frac{(\prod_{i}\mu(n_i))c'(N,T)}{[F_{N,T}:K]}+ O\,\Bigg(\sum_{\substack{N\\ n_i\in E}}\frac{\sqrt{x}\log\left(x^{[F_{N,T}:\Q]}\cdot\av{d_{F_{N,T}}} \right)}{[F_{N,T}:K]}  \Bigg).
\end{equation}
We can decompose the  $O$-term in two parts, the first one being
\[ O\,\Bigg(\sqrt{x}\log x\cdot\sum_{\substack{N\\ n_i\in E}}1\Bigg)
	=\bo{\sqrt{x}\log x\av{E}^r}=O\left(x^{2/3}\log x\right)\,, \]
as we have $\av{E}\leq 2^{\pi(\xi)}\leq e^\xi=x^{1/6r}$, where $\pi$ is the prime counting function. In particular, this shows that the error term in \eqref{eq-m} includes $\bo{\sqrt{x}/\log x}$.

As for the second part of the $O$-term, applying Corollary \ref{cor-bound} first and then Proposition \ref{prop-discr}, we obtain
\begin{align*}
& \ O\left(\sqrt{x}\sum_{\substack{N\\ n_i\in E}}\frac{\log\av{d_{F_{N,T}}}}{\varphi([n_1t_1,\ldots,n_rt_r])n_1t_1\cdots n_rt_r}  \right)  \\
=&\ \bo{\sqrt{x}\sum_{\substack{N\\ n_i\in E}}\Big(\log([n_1t_1,\ldots,n_rt_r]n_1t_1\cdots n_rt_r)+\bo{1}\Big)}  \\
=&\ \bo{\sqrt{x}\cdot 2\sum_{\substack{N\\ n_i\in E}}\bigg(\sum_{j=1}^r\log n_j\bigg)+\sqrt{x}\cdot2\sum_{\substack{N\\ n_i\in E}}\bigg(\sum_{j=1}^r\log t_j\bigg)}  \\
=&\ \bo{\sqrt{x}\av{E}^{r-1}\cdot\sum_{k\in E}\log k} +\bo{ \sqrt{x}\av{E}^{r}\cdot\log\big( \max_i (t_i)\big)}.
\end{align*}
By assumption we have $\log t_i=\bo{\log x}$ for all $i$, whereas 
since the largest integer in $E$ is $\prod_{q<\xi}q$, where $q$ runs through rational primes, we have
\[ \sum_{k\in E}\log k\leq \av{E}\cdot\sum_{q<\xi}\log q\leq\av{E}\cdot\xi \leq x^{1/6r}\log x, \]
where the second inequality follows by \cite[Theorem 415]{hawr}, and the last one by recalling that $\av{E}\leq x^{1/6r}$ and $\xi\leq\log x$.
Thus, making use of these estimates, also these error terms are reduced to $\bo{x^{2/3}\log x}$.

We now focus on the main term of \eqref{eq-m} and we will estimate the tail of the series as
\[ \Bigg\lvert\sum_{\substack{N\\ n_i\notin E \text{ for some }i}}\frac{\prod_i\mu(n_i)c'(N,T)}{[F_{N,T}:K]}\Bigg\rvert\leq
\sum_{\substack{N\\ n_i\notin E' \text{ for some }i}}\frac{c'(N,T)}{[F_{N,T}:K]}\,, \]
where $E'$ is the set of all positive integers  whose prime factors $q$ satisfy $q<\xi$. Then we bound the latter series of nonnegative terms by
\[  \Bigg(\sum_{\substack{N \\ n_1\notin E' }}
+ \ldots+\sum_{\substack{N \\ n_r\not\in E' }}\Bigg) \frac{c'(N,T)}{[F_{N,T}:K]}
\leq \Bigg(\sum_{\substack{N \\ n_1\geq\xi}}
+ \ldots+\sum_{\substack{N \\ n_r\geq\xi}}\Bigg)\frac{c'(N,T)}{[F_{N,T}:K]}\,.  \]
Since $c'(N,T)\leq [F:K]$, applying Corollary \ref{cor-bound} and Theorem \ref{pollack}, we can estimate each series by $\bo{1/\xi}=\bo{1/\log x}$. Using that $\Li(x)=\bo{x/\log x}$ and summing up all the errors, we obtain $O(x/\log^2x)$.
Finally, because of the formula $\Li(x)=x/\log x+O(x/\log^2x)$, we can replace $\Li(x)$ with $x/\log x$ in the the main term of $\mc M_\xi(x)$ as the multiple series converges by Corollary \ref{cor-convergence1}.
\end{proof}

Let us now focus on the error term of \eqref{eq4}. 

\begin{lem}\label{lem-index-error}
Assume (GRH). Let $x\geq t_i^3$ for all $i$.  Then we have
\[ \mc M_{\xi,\eta}(x)=\bo{\frac{x}{\log^2x}}+\bo{\frac{x\log\log x}{\log^2x}\sum_{i=1}^r\frac{1}{\varphi(t_i)}}\,. \]
\end{lem}

\begin{proof}
We can bound $\mc M_{\xi,\eta}(x)$ by 
\[
\sum_{i=1}^r\cset{\p: \N\p\leq x,\, t_iq\mid\ind_\p(\alpha_i)\text{ for some prime } \xi\leq q<\eta, \begin{pmatrix}\p \\ F/K  \end{pmatrix}\subseteq C}\,,
\]
and we can conclude directly because each of these terms can be bounded by the sum of three errors (see \cite[page 73]{zieg}) which are estimated in \cite[Lemmas 9, 10, 11]{zieg}. Notice that our different value for $\xi$ does not change the proof of \cite[Lemma 11]{zieg}, whereas \cite[Lemmas 9, 10]{zieg} do not depend on $\xi$. 

Moreover, it is straightforward to see that these Lemmas hold also for algebraic numbers (see also \cite[Proof of Proposition 5.1]{PS1}). In fact, in \cite[Proof of Lemma 9]{zieg}, if $\alpha=\beta/\gamma$ with $\beta,\gamma\in\OO_K$, then the congruence $\alpha^{(\N\p-1)/tq}\equiv 1\bmod\p$ yields the inclusion of integral ideals $\p\supseteq(\beta^{(\N\p-1)/tq}-\gamma^{(\N\p-1)/tq})$, and then proceeding with the original proof we set $A$ to be the maximum of $\{1\}\cup\{\av{\sigma(\beta)},\av{\sigma(\gamma)}:\sigma\in\Gal(K/\Q)\}$.
\end{proof}

\begin{proof}[Proof of Theorem \ref{thm-index1}]
The statement follows by invoking formula \eqref{eq4} and  applying Lemmas \ref{lem-index-main} and \ref{lem-index-error}.
\end{proof}

\section{Setting conditions on the index}\label{sec-index2}
In this section we prove Theorem \ref{thm-index}, keeping the notation of the introduction.
The following result is a variant of \cite[Lemma 13]{zieg}.
\begin{lem}\label{lem-index}
Assume (GRH). Let $\gamma$ be a nonzero algebraic number of $K$ which is not a root of unity. Let $0<\rho<1$. We have that
\[ \cset{\p:\N\p\leq x,\,\ind_\p(\gamma)>(\log x)^{\rho}}=\bo{\frac{x}{(\log x)^{1+\rho}}}+\bo{\frac{x}{(\log x)^{2-\rho} }}\,. \]
\end{lem}
Notice that we are discarding the finitely many primes $\p$ of $K$ appearing in the factorization of $\gamma$.

\begin{proof}
We only point out the modification with respect to the proof of \cite[Lemma 13]{zieg} (where $\rho=1/2$).
Let $y:=\floor{(\log x)^{\rho}}$. Following the original proof, the error terms that we obtain are:
\[ \bo{\frac{xy}{\log^{2} x}}=\bo{\frac{x}{(\log x)^{2-\rho} }} \]
\[ \bo{\frac{x}{y\log x}}=\bo{\frac{x}{(\log x)^{1+\rho}}}\,. \]
Notice that  $\bo{\sqrt{x}/\log x}$ is included in these error terms.
\end{proof}

\begin{proof}[Proof of Theorem \ref{thm-index}]
We take $0<\rho\leq1/2$, so that the set considered in the previous Lemma has size $O(x/(\log x)^{1+\rho})$. Write $y:=(\log x)^{\rho}$, and write $\mc R_T$ for the set $\mc R$ in \eqref{r-set} to make the dependence on the $r$-tuple $T$ explicit. Then we can partition the set $\mc S$ as the disjoint union of all sets $\mc R_T$ with $t_i\in S_i$ for all $i$.  We have
\begin{equation}\label{eq-tail}
\sum_{\substack{T \\ t_i\in S_i }}\mc R_{T}(x)-\sum_{\substack{T \\ t_i\leq y,\,t_i\in S_i}}\mc R_{T}(x)\leq \sum_{\substack{ T\\ t_1>y }}\mc R_{T}(x)+\ldots+
\sum_{\substack{ T \\ t_r>y}}\mc R_{T}(x)\,.
\end{equation}
Applying Lemma \ref{lem-index}, each multiple series on the right-hand side can be bounded by
\[ \cset{\p:\N\p\leq x,\,\ind_\p(\alpha_i)>y}=\bo{\frac{x}{(\log x)^{1+\rho}}}\,, \]
respectively. This yields the formula
\begin{equation}\label{s-r-formula}
\mc S(x)=\sum_{\substack{T \\ t_i\leq y,\,t_i\in S_i}}\mc R_{T}(x)+\bo{\frac{x}{(\log x)^{1+\rho}}}\,.
\end{equation}

We now replace the asymptotic \eqref{r-asymptotic} for the functions $\mc R_T(x)$ in \eqref{s-r-formula}, as $x>y^3$. Let us first focus on the main term that we obtain, namely
\begin{equation}\label{multisum0}
\frac{x}{\log x}
	\sum_{\substack{T \\ t_i\leq y,\, t_i\in S_i }}
	\sum_{N}\frac{(\prod_{i}\mu(n_i))c'(N,T)}{[F_{N,T}:K]}\,.
\end{equation}
Call $D_T(x)$ the (inner) multiple series on the multi-indices $N$ appearing in \eqref{multisum0} and notice that $D_T(x)\geq0$. Similarly to \eqref{eq-tail}, we have
\begin{equation}\label{multisum1}
\sum_{\substack{T \\ t_i\in S_i }}D_T(x)-\sum_{\substack{T \\ t_i\leq y,\, t_i\in S_i}}D_T(x)\leq \sum_{\substack{ T \\ t_1>y }}D_T(x)+\ldots+
\sum_{\substack{ T \\ t_r>y }}D_T(x)\,.
\end{equation}
Since $[n_1t_1,\ldots,t_rn_r]$ is a multiple of $[t_1,n_1,\ldots,t_r,n_r]$,
applying Corollary \ref{cor-bound} we have
\[ \frac{1}{[F_{N,T}:F]}\leq\frac{B[F:\Q]}{\varphi([t_1,n_1,\ldots,t_r,n_r])t_1n_1\cdots t_rn_r}\,, \]
and hence by Theorem \ref{pollack} each series on the right-hand side of \eqref{multisum1}, multiplied by $x/\log x$, has size (recalling $c'(N,T)\leq [F:K]$)
\[ \bo{\frac{x}{y\log x}}=\bo{\frac{x}{(\log x)^{1+\rho}}}\,. \]

Next we study the error terms that we obtain when replacing $\mc R_T(x)$ in \eqref{s-r-formula}.  The first part of the $O$-term of \eqref{r-asymptotic} gives
\[ \bo{\frac{xy^r}{\log^{2}x}}=\bo{\frac{x}{(\log x)^{2-\rho r}}}\,, \]
where we suppose that $\rho$ satisfies $2-\rho r>0$. Recall the formula $\sum_{k<y}1/\varphi(k)=\bo{\log y}$ (see for instance \cite[Lemma 8]{zieg}). Then, for each $j\in\{1,\ldots,r\}$, the second part of the $O$-term of \eqref{r-asymptotic} yields the sum of errors
\begin{align*}
&\ O \Bigg(\frac{x\log\log x}{\log^2x}\sum_{\substack{T \\ t_i\leq y }}\frac{1}{\varphi(t_j)}\Bigg)
= \bo{\frac{x\log\log x}{\log^2 x}\cdot y^{r-1}\log y} \\
= & \ \bo{\frac{x(\log\log x)^2}{(\log x)^{2-\rho(r-1)}}}=\bo{\frac{x}{(\log x)^{a}}\frac{(\log\log x)^2}{(\log x)^{b}}}
\end{align*}
where we   take $a,b>0$ such that $a+b=2-\rho(r-1)$. Note that $b$ can be chosen arbitrarily small, because $(\log\log x)^2/(\log x)^{b}$ tends to zero as $x\rightarrow\infty$ for any $b>0$. Therefore, for $j\in\{1,\ldots,r\}$, this error term becomes
$O(x/(\log x)^{a})$ for some $0<a<2-\rho(r-1)$ and thus can be included in $O(x/(\log x)^{2-\rho r})$. 

It remains to choose a suitable value for $\rho$. In the  $O$-terms we have the exponents $1+\rho$ and $2-\rho r$.
A possible choice is $\rho=1/(r+1)$, yielding the error term
\[ \bo{\frac{x}{(\log x)^{1+\frac{1}{r+1}}}}\,.  \]
This concludes the proof.
\end{proof}

\section{The asymptotic formula for the order}\label{sec-order}
In this section we obtain an asymptotic formula for the function $\mc P(x)$ of Theorem \ref{main} by expressing it as a sum  of functions of the type $\mc R(x)$. Let us keep the notation of the introduction.

Let $\p\in \mc P$ be of degree $1$, and call $p$ the rational prime below $\p$. Because of the identity $\ord_\p(\alpha_i)\cdot\ind_\p(\alpha_i)= \N \p-1$, the condition $\ord_{\p}(\alpha_i)\equiv a_i\bmod d_i$ is  equivalent to $\ind_\p(\alpha_i)=t_i$ and $p\equiv 1+a_it_i\bmod d_it_i$. We define the sets of primes of $K$ satisfying these conditions by setting
\[ \mc V_{T}:= \left\{ \p : \forall i \
\ind_\p(\alpha_i)=t_i,\, p\equiv 1+a_it_i\bmod d_it_i, \,
	\begin{pmatrix}\p \\ F/K  \end{pmatrix}\subseteq C\right\}\,. \]
Notice that the sets $\mc V_T$ give a partition of $\mc P$  as the multi-index $T$  varies, up to discarding the primes which are not of degree $1$. Thus we have
\begin{equation}
\mc P(x)=\sum_{T}\mc V_{T}(x)+\bo{\frac{\sqrt{x}}{\log x}} \,.
\label{sumpg}
\end{equation}
Given $T$ such that $1+a_it_i$ and $d_i$ are not coprime for some $i$,   the set $\mc V_T$ contains at most $[K:\Q]$ primes, as in this case a prime $\p\in\mc V_T$ must lie above a fixed prime divisor of $d_i$. 
There are finitely many primes $\p$ lying in some  $\mc V_T$ with  $1+a_it_i$ and $d_i$ not coprime for some $i$ (at most $[K:\Q]$ primes $\p$ for each rational prime $p$ dividing one of the integers $d_i$), and since the sets $\mc V_T$ are disjoint, they are counted only once.
Therefore, we may restrict the multiple series in \eqref{sumpg} to the multi-indices $T$ with $(1+a_it_i,d_i)=1$ for every $i$.

Recall the notation $w=w(T):=[d_1t_1,\ldots,d_rt_r]$.
\begin{lem}\label{lem-v}
Fix $T$ such that $(1+a_it_i,d_i)=1$ for all $i$. Then the set $\mc V_T$ can be written as
\[ \mc V_{T}=\set{\p: \ind_\p(\alpha_i)=t_i\,\forall  i,\, \begin{pmatrix}\p \\ F(\zeta_{w})/K  \end{pmatrix}\subseteq C_{w}}\,, \]
where 
\[ C_{w}:=\left\{ \sigma\in\Gal(F(\zeta_w)/K) : \forall i \ \sigma(\zeta_{d_it_i})=\zeta_{d_it_i}^{1+a_it_i},\, \sigma|_F\in C\right\}\,. \]
Moreover, assuming (GRH) and letting $x\geq t_i^{3}$ for all $i$, the function $\mc V_T(x)$ satisfies
\begin{equation}\label{v-fct}
\mc V_T(x)=\frac{x}{\log x}\sum_{N}\frac{(\prod_{i}\mu(n_i))c(N,T)}{[F_{w,N,T}:K]} +\bo{\frac{x}{\log^{2}x}+\sum_{i=1}^r\frac{x\log\log x}{\varphi(t_i)\log^2x}},
\end{equation}
where $F_{w,N,T}$ and $c(N,T)$ are as in Theorem \ref{main}. Moreover, $c(N,T)>0$ holds only if we have $(d_i,n_i)\mid a_i$ for all $i$.
The constant implied by the $O$-term depends only on $K$, $F$, the $\alpha_i$'s and the $d_i$'s.
\end{lem}

\begin{proof}
We keep the notation and the assumptions described above.
Since $1+a_it_i$ and $d_i$ are coprime for all $i$, if $\p\in\mc V_T$ and $\N\p=p$, then $p\nmid d_it_i$ for all $i$ and we have the equivalence
\begin{equation}\label{equivalence}
p\equiv 1+a_it_i\bmod d_it_i \quad \Longleftrightarrow \quad 
\begin{pmatrix}p \\ \Q(\zeta_{d_it_i})/\Q  \end{pmatrix} \text{ satisfies } \zeta_{d_it_i}\mapsto\zeta_{d_it_i}^{1+a_it_i}\,.
\end{equation}
The first part of the statement is now clear, because the condition on the right of \eqref{equivalence} holds for all $i$ if and only if $(\p,K(\zeta_w)/K)$ acts as the exponentiation by $1+a_it_i$ on $\zeta_{d_it_i}$ for all $i$. 

In order to get an asymptotic formula for $\mc V_T(x)$, it is sufficient to apply Theorem \ref{thm-index1}  to the field extension $F(\zeta_w)/K$ and  $C_w$. 
Notice that in this application of Theorem \ref{thm-index1} the number $c'(N,T)$ coincides with the number $c(N,T)$ of Theorem \ref{main}. Moreover, the coefficient $c(N,T)$ is zero if $(d_i,n_i)\nmid a_i$ for some $i$ because if an automorphism $\sigma$ is counted, then it must act on $\zeta_{(d_i,n_i)t_i}$ as the identity and as the exponentiation by $1+a_it_i$.

As for the constant implied by the $O$-term, one can check easily that applying Theorem \ref{thm-index1} to $F(\zeta_w)/K$ and  $C_w$ preserves  its independence from the parameters $t_i$ (except for the factors $\varphi(t_i)$, which are already explicit). Indeed, in the proof of Lemma \ref{lem-index-main} it is sufficient to take into account the bound  $c(N,T)\leq[F:K]$, and to see $F_{w,N,T}$ as a cyclotomic-Kummer extension of $F$ when applying Proposition \ref{prop-discr}.
\end{proof}

We are now ready to prove Theorem \ref{main}.
\begin{proof}[Proof of Theorem \ref{main}]
Let $\mc V_T$ be as above.
We follow the proof of Theorem \ref{thm-index} closely. Take $0<\rho\leq1/2$ and set $y:=(\log x)^{\rho}$. Consider formula \eqref{sumpg}.
We estimate the tail of the function $\mc P(x)$ in the same way as we did for the function $\mc S(x)$, i.e.\ similarly as in \eqref{eq-tail} and then applying Lemma \ref{lem-index}  (we may take $S_i$ to be the set of all positive integers for all $i$). We obtain 
\begin{equation}
\mc P(x)=\sum_{\substack{T \\ t_i\leq y}}\mc V_{T}(x)+
		\bo{\frac{x}{(\log x)^{1+\rho}}}\,,
\label{eqpro3}
\end{equation}
where we may restrict the indices $t_i$ to those satisfying $(1+a_it_i,d_i)=1$ for all $i$. Notice that $\bo{\sqrt{x}/\log x}$ is also included in the error term.

We choose $\rho=1/(r+1)$.
As $x>y^3$, by Lemma \ref{lem-v} we may replace in \eqref{eqpro3} $\mc V_T(x)$ by the asymptotic \eqref{v-fct}. Let us first focus on the main term, namely
\[ \frac{x}{\log x}
	\sum_{\substack{T \\ t_i\leq y} }\
	\sum_{N}
	\frac{(\prod_{i}\mu(n_i))c(N,T)}{[F_{w,N,T}:K]}\,, \]
	where we may restrict the indices $n_i$ to those with $(d_i,n_i)\mid a_i$. We deal with the multiple  sum on  the multi-indices $T$ as we did in \eqref{multisum1} (where the condition $t_i\in S_i$ trivially holds). Since the degree $[F_{w,N,T}:K]$ is a multiple of $[F_{N,T}:K]$, we can then proceed as in the proof of Theorem \ref{thm-index}, i.e.\ by applying Theorem \ref{pollack} (recalling that $c(N,T)\leq[F:K]$).  Finally, we control the error terms directly as we did for $\mc S(x)$.
\end{proof}

Notice that in the case $r=1$, our choice for $\rho$ yields the same error obtained by Ziegler in \cite[Theorem 1]{zieg}.

\begin{rem}
For  $N,T$ fixed, we could say more about the necessary conditions for the coefficient $c(N,T)$ to be  nonzero. Suppose it counts at least one element $\sigma\in\Gal(F_{w,N,T}/K)$. Then for each $i$, $\sigma$ must act as the exponentiation by $1+a_it_i$ on the root of unity $\zeta_{d_it_i}$, so that the system of congruences $y\equiv a_it_i\bmod d_it_i$ must be solvable. This is the case if and only if we have $a_it_i\equiv a_jt_j\bmod(d_it_i,d_jt_j)$ for every $i,j$, and the solution $y$ will be unique modulo $w$. This integer $y=y(T)$  would be such that  $[t_1,\ldots,t_r]\mid y$ and $\sigma(\zeta_w)=\zeta_w^{1+y}$.

Suppose that the above-mentioned system has a solution $y$. Then the element $\tau\in\Gal(\Q(\zeta_w)/\Q)$ such that $\tau(\zeta_w)=\zeta_w^{1+y}$ must be the identity on  $\Q(\zeta_w)\cap K_{N,T}$. This implies that $\tau$ fixes $\zeta_{(w,v)}$ where $v=v(N,T):=[n_1t_1,\ldots,n_rt_r]$, so that we must have $(w,v)\mid y$. This condition implies for instance that $(d_it_i,n_jt_j)\mid a_it_i$ for every $i,j\in\{1,\ldots,r\}$ (because $y\equiv a_it_i\bmod d_it_i$).
\end{rem}

\begin{rem}
Let $K$ be a number field, and let $G_1,\ldots,G_r$ be finitely generated subgroups of $K^\times$ of finite positive rank $s_1,\ldots,s_r$, respectively, which generate a torsion-free subgroup of $K^\times$ of rank $\sum_i s_i$. For a prime $\p$ of $K$, let $\ord_\p(G_i)$ be the order of the reduction of  $G_i$ modulo $\p$, when this is well-defined.
 In Theorem \ref{main}, one could instead study the set of primes $\p$ of $K$ satisfying $\ord_\p(G_i)\equiv a_i\bmod d_i$ for all i (and possibly an additional condition on the Frobenius). The result would be analogous, simply replacing $\alpha_i$ with $G_i$ in the definitions of $F_{w,N,T}$ and $c(N,T)$. The error term would be the same, i.e.\ with the exponent $(1+1/(r+1))$ in  the denominator.

Indeed, the author and Perucca generalized Ziegler's results \cite{zieg} to finite rank in \cite{PS1}, so that one could use the latter work to achieve directly all the steps of the present paper for the problem introduced in this remark.  

Write $\ind_\p(G_i)$ for the index of the reduction of $G_i$ modulo $\p$. One could also study the sets analogous to those of Theorems \ref{thm-index1} and \ref{thm-index} with the conditions $\ind_\p(G_i)=t_i$ and $\ind_\p(G_i)\in S_i$, respectively. The analogous results can be obtained also in this case.
\end{rem}

\section*{Acknowledgments}
The author would like to thank Antonella Perucca for suggesting the problem and for her valuable feedback. A special thank to Paul Pollack for his considerable help with the results of Section \ref{sec-euler}.

\end{document}